\documentclass[leqno,11pt]{amsart}
\usepackage{amsmath,amstext,amssymb,amsopn,amsthm,mathrsfs}

\allowdisplaybreaks

\DeclareMathOperator{\domain}{Dom}

\DeclareMathOperator{\spann}{span}

\newtheorem{theor}{Theorem}[section]
\newtheorem{propo}[theor]{Proposition}
\newtheorem{lemma}[theor]{Lemma}

\newtheorem*{rem*}{Remark}

\newtheorem{conj}[theor]{Conjecture}

\begin{document}

\title[Heat kernel bounds in the F-B setting]
    {On sharp heat and subordinated kernel estimates\\ in the Fourier-Bessel setting}


\author[A. Nowak]{Adam Nowak}
\author[L. Roncal]{Luz Roncal}

\address{Adam Nowak, \newline
			Instytut Matematyczny,
      Polska Akademia Nauk, \newline
      \'Sniadeckich 8,
      00--956 Warszawa, Poland
      }
\email{adam.nowak@impan.pl}
\thanks{Research of the first-named author supported by MNiSW Grant N N201 417839.
\\ \indent Research of the second-named author supported by the grant MTM2009-12740-C03-03 of the DGI}

\address{Luz Roncal, \newline Departamento de Matem\'aticas y Computaci\'on,
	Universidad de La Rioja, \newline Edificio J.L.~Vives, Calle Luis
	de Ulloa s/n, 26004 Logro\~no, Spain}
\email{luz.roncal@unirioja.es}

\begin{abstract}
We prove qualitatively sharp heat kernel bounds in the setting of Fourier-Bessel expansions
when the associated type parameter $\nu$ is half-integer.
Moreover, still for half-integer~$\nu$, we also obtain sharp estimates of all kernels
subordinated to the heat kernel. Analogous estimates for general
$\nu > -1$ are conjectured. Some consequences concerning the related
heat semigroup maximal operator are discussed.
\end{abstract}

\maketitle

\section{Introduction} \label{sec:intro}

Let $J_{\nu}$ denote the Bessel function of the first kind and order $\nu$ and let
$\{\lambda_{n,\nu} : n \ge 1\}$ be the sequence of successive positive zeros of $J_{\nu}$
in increasing order. It is well known that for each fixed $\nu >-1$ the functions
$$
\phi_n^{\nu}(x) = d_{n,\nu}\, x^{-\nu-1\slash 2} (\lambda_{n,\nu}x)^{1\slash 2}
J_{\nu}(\lambda_{n,\nu}x), \qquad n=1,2,\ldots,
$$
form an orthonormal basis in $L^2((0,1),d\mu_{\nu})$; here
$d_{n,\nu} = \sqrt{2}|\lambda_{n,\nu}^{1\slash 2}J_{\nu+1}(\lambda_{n,\nu})|^{-1}$ are
normalizing constants and the measure is given by
$$
d\mu_{\nu}(x) = x^{2\nu+1}\,dx.
$$
The system $\{\phi_{n}^{\nu} : n=1,2,\ldots\}$ is usually referred to as the Fourier-Bessel system.
Another Fourier-Bessel system arises by considering the functions
$$
\psi_n^{\nu}(x) = x^{\nu+1\slash 2} \phi_n^{\nu}(x), \qquad n=1,2,\ldots,
$$
which form an orthonormal basis in $L^2((0,1),dx)$.
Fourier-Bessel expansions are of interest and exist in the literature for a long time, see
\cite[Chapter XVIII]{Wat}. The related convergence problems were investigated in
\cite{BP, BPconv, Wing}. For the study of several fundamental harmonic analysis operators in the
Fourier-Bessel context see \cite{CiauRo, CiauRon, CRae, CR, OS, OScon, Muc-St} and references therein.
Some interesting open questions concerning harmonic analysis of Fourier-Bessel expansions can be found in
a recent paper by Betancor \cite{Bet}.

Given $\nu>-1$, consider the differential operator
$$
L_{\nu} = -\Delta - \frac{2\nu+1}{x} \, \frac{d}{dx},
$$
which is symmetric and nonnegative on $C_c^2(0,1) \subset L^2((0,1),d\mu_{\nu})$.
Each $\phi_n^{\nu}$, $n=1,2,\ldots$, is an eigenfunction of $L_{\nu}$ with the corresponding eigenvalue
$\lambda_{n,\nu}^2$,
$$
L_{\nu}\phi_n^{\nu} = \lambda_{n,\nu}^2 \phi_n^{\nu}.
$$
This leads to a natural self-adjoint extension of $L_{\nu}$ given by
$$
\mathcal{L}_{\nu}f = \sum_{n=1}^{\infty} \lambda_{n,\nu}^2 \, \langle f, \phi_n^{\nu}\rangle_{d\mu_{\nu}}\,
	\phi_n^{\nu}
$$
on the domain $\domain \mathcal{L}_{\nu}$ consisting of all $f \in L^2((0,1),d\mu_{\nu})$ for which
the series converges in $L^2((0,1),d\mu_{\nu})$;
here $\langle f, \phi_n^{\nu}\rangle_{d\mu_{\nu}}$ is the $n$th Fourier-Bessel coefficient of $f$.
Clearly, the spectral decomposition of
$\mathcal{L}_{\nu}$ is given by the $\phi_n^{\nu}$.

The aim of this paper is to find sharp estimates for the kernels
\begin{equation} \label{ker}
G_t^{\nu,\alpha}(x,y) = \sum_{n=1}^{\infty} \exp\big(-t \lambda_{n,\nu}^{\alpha}\big)
 \phi_n^{\nu}(x)\phi_n^{\nu}(y), \qquad x,y \in (0,1), \quad t>0,
\end{equation}
where $\alpha \in (0,2]$ is the index of subordination.
These are precisely the integral kernels of the semigroups
$$
\big\{T_t^{\nu,\alpha}\big\}_{t>0} := \big\{ \exp(-t\mathcal{L}_{\nu}^{\alpha\slash 2})\big\}_{t>0}
$$
generated either by $\mathcal{L}_{\nu}$ itself ($\alpha=2$) or by its fractional powers ($\alpha<2$).
Needless to say, the cases $\alpha=2,1$ are of prior importance since then $\eqref{ker}$ becomes
the heat or the Poisson kernel, respectively, related to the system $\{\phi_n^{\nu}\}$.
The behavior of $G_t^{\nu,\alpha}(x,y)$ does not seem to have been studied before.
Note that for short times $t$ a direct analytic treatment of the series in
\eqref{ker} is a complicated matter except for a few particular cases discussed in Section \ref{sec:rem}.
First of all, the series defining $G_t^{\nu,\alpha}(x,y)$ is highly oscillating, and its behavior
is hidden behind subtle cancellations between the oscillations. Moreover, $J_{\nu}$ is in general
a transcendental function and there is no explicit formula for its zeros.
Finally, the power $\alpha$ of the eigenvalues in the argument of the exponential makes the situation even
more sophisticated. These obstacles can be better understood by means of the explicit computations
in Section \ref{sec:rem} for the cases $\nu={\pm 1\slash 2}$, $\alpha =1,2$.

Our method of estimating $G_t^{\nu,\alpha}(x,y)$ is based on a connection between the Fourier-Bessel
context and the classical setting related to multi-dimensional Euclidean balls.
More precisely, we show that the fundamental solution $\mathcal{G}_t^{(d),2}(\mathtt{x},\mathtt{y})$
of the classical heat equation in the $d$-dimensional unit ball,
subject to the Dirichlet boundary condition, is related to $G_t^{\nu,2}(x,y)$ with
$\nu=d\slash 2 -1$. This allows us to transfer known bounds for
$\mathcal{G}_t^{(d),2}(\mathtt{x},\mathtt{y})$
to the Fourier-Bessel framework on the interval $(0,1)$ and conclude qualitatively sharp estimates
for $G_t^{\nu,2}(x,y)$ when $\nu$ is half-integer, see Theorem \ref{thm:main}.
Essentially the same procedure applies to the subordinated kernels
$\mathcal{G}_t^{(d),\alpha}(\mathtt{x},\mathtt{y})$ and $G_t^{\nu,\alpha}(x,y)$,
see Theorem \ref{thm:poisson}. We emphasize that the sharp bounds for
$\mathcal{G}_t^{(d),\alpha}(\mathtt{x},\mathtt{y})$, $0<\alpha \le 2$, are quite strong results
and seem to have been obtained in a complete form only recently \cite{Z, Song};
this may be a bit surprising, in view of simplicity of the geometry of the underlying domain,
the Euclidean ball.
For large values of $t$ sharp estimates of $G_t^{\nu,\alpha}(x,y)$ can be derived directly
from the series representation \eqref{ker}, thanks to the exponential decay. In this case all
$\nu >-1$ are covered, see Theorem \ref{thm:large}.
We remark that all the results can be immediately translated to the setting of the system
$\{\psi_n^{\nu}\}$ since the relevant kernels coincide up to the factor $(xy)^{\nu+1\slash 2}$.

The paper is organized as follows. In Section \ref{sec:balls} we first discuss the connection between the
Fourier-Bessel setting in the interval $(0,1)$ and the situation associated with Euclidean balls,
and then relate the kernels $G_t^{\nu,\alpha}(x,y)$ and $\mathcal{G}_t^{(d),\alpha}(\mathtt{x},\mathtt{y})$.
In Section \ref{sec:heat} we invoke necessary bounds for
$\mathcal{G}_t^{(d),\alpha}(\mathtt{x},\mathtt{y})$ and transfer them to the Fourier-Bessel
setting in the interval. We also establish the long time behavior of $G_t^{\nu,\alpha}(x,y)$
by analyzing the defining series.
Finally, Section \ref{sec:rem} is devoted to various comments and remarks.
These concern, in particular, mapping properties
of the maximal operators of the semigroups $\{T_t^{\nu,\alpha}\}_{t>0}$, $0<\alpha\le 2$.

Throughout the paper we use a standard notation.
While writing estimates, we will use the notation $X \lesssim Y$ to
indicate that $X \le CY$ with a positive constant $C$ independent of significant quantities. We shall
write $X \simeq Y$ when simultaneously $X \lesssim Y$ and $Y \lesssim X$.

\subsection*{Acknowledgment}
This research was started during the sojourn in Wroc\l{}aw of the second-named author during the 
winter of 2011. She acknowledges a financial support received from Politechnika Wroc\l{}awska 
and wants to thank for the hospitality offered.

\section{Connection with multi-dimensional balls} \label{sec:balls}

For $d \ge 1$, let $B^d=\{\mathtt{x} \in \mathbb{R}^d : |\mathtt{x}|<1\}$ be the unit ball in $\mathbb{R}^d$
and denote $S^{d-1}=\partial B^d$. It is well known (see for instance \cite[Chapter 2, H]{Folland})
that there exists an orthonormal basis of eigenfunctions associated with the Dirichlet Laplacian
in $B^d$. These are expressed explicitly by the functions $\phi_n^{\nu}$ and spherical harmonics,
as described below. For more details on spherical harmonics and their connections with symmetry properties
of the Fourier transform we refer to \cite[Chapter IV]{SW}.

Given $k \ge 0$, let $\mathcal{P}_k$ be the space of homogeneous polynomials of degree $k$ in
$\mathbb{R}^d$. Taking harmonic polynomials in $\mathcal{P}_k$ and restricting them to $S^{d-1}$ we
obtain the space
$$
    H_k =\big\{P|_{S^{d-1}}:P\in \mathcal{P}_k \textrm{ and }\Delta P=0\big\}.
$$
The elements of $H_k$ are called spherical harmonics of degree $k$. The dimension of $H_k$ is finite,
in fact we have
$$
\mathfrak{d}_k=\dim H_k=(2k+d-2)\frac{(k+d-3)!}{k!(d-2)!}, \qquad d \ge 2.
$$
The case $d=1$ is degenerated: $S^0$ consists of two points and so $\mathfrak{d}_0=\mathfrak{d}_1=1$
and $\mathfrak{d}_k=0$ for $k>1$.

Let $\{Y_m^k : 1 \le m \le \mathfrak{d}_k\}$ be an orthonormal basis for $H_k$ in $L^2(S^{d-1},d\sigma)$,
where $\sigma$ is the standard (non-normalized) surface area measure on $S^{d-1}$. Define the functions
$$
\Phi_{n,k,m}^{(d)}(\mathtt{x}) = \phi_{n}^{k+d\slash 2-1}\big(|\mathtt{x}|\big)
	Y_m^k\left( \frac{\mathtt{x}}{|\mathtt{x}|}\right), \qquad n \ge 1, \quad k \ge 0, \quad
		1\le m \le \mathfrak{d}_k;
$$
here and later on we tacitly assume that for $d=1$ only $k=0,1$ are considered.
The system
$$
\big\{ \Phi_{n,k,m}^{(d)} : n \ge 1, k \ge 0, 1\le m \le \mathfrak{d}_k \big\}
$$
is orthonormal and complete in $L^2(B^d,d\mathtt{x})$. Moreover, it consists of eigenfunctions
of the Dirichlet Laplacian in $B^d$,
$$
\Delta \Phi_{n,k,m}^{(d)} = -\lambda_{n,k+d\slash 2 -1}^2 \Phi_{n,k,m}^{(d)}, \qquad
	\Phi_{n,k,m}^{(d)}\Big|_{S^{d-1}} = 0.
$$
Thus the associated heat kernel and the subordinated kernels are expressed as
$$
\mathcal{G}_t^{(d),\alpha}(\mathtt{x},\mathtt{y}) = \sum_{n \ge 1} \sum_{k \ge 0}
	\sum_{1 \le m \le \mathfrak{d}_k} \exp\big({-t\lambda_{n,k+d\slash 2-1}^{\alpha}}\big)
	\Phi_{n,k,m}^{(d)}(\mathtt{x}) \Phi_{n,k,m}^{(d)}(\mathtt{y}), \qquad \mathtt{x},\mathtt{y} \in B^d,
		\quad t>0,
$$
where $\alpha \in (0,2]$ is the subordination index.
These are the integral kernels of the semigroups
$$
\big\{ \mathcal{T}_t^{(d),\alpha}\big\}_{t>0} :=
	\big\{ \exp\big(-t(-\Delta)^{\alpha\slash 2}\big)\big\}_{t>0},
$$
where $-\Delta$ is understood to be the nonnegative self-adjoint operator in $L^2(B^d,d\mathtt{x})$
whose spectral resolution is given by the $\Phi_{n,k,m}^{(d)}$. The domain of this operator can be
identified with the Sobolev space $H_0^1(B^d)$, cf. \cite{Ba-Co}.

We now observe that the analysis of the radial case in the context of expansions with respect
to $\{\Phi_{n,k,m}^{(d)}\}$ reduces to the Fourier-Bessel setting in the interval $(0,1)$ with
the type index $\nu=d\slash 2-1$. Indeed, if $F(\mathtt{x})=f(|\mathtt{x}|)$ is radial,
then integrating in polar coordinates and taking into account that
$Y_1^0 \equiv (\sigma(S^{d-1}))^{-1\slash 2}$ we see that
$$
\langle F, \Phi_{n,0,1}^{(d)}\rangle =
 \int_{B^d} F(\mathtt{x}) \Phi_{n,0,1}^{(d)}(\mathtt{x})\, d\mathtt{x}
	= c_d \int_0^1 f(x) \phi_n^{d\slash 2-1}(x)\, x^{d-1}\, dx
	= c_d \langle f, \phi_n^{d\slash 2-1}\rangle_{d\mu_{d\slash 2-1}},
$$
where $c_d=(\sigma(S^{d-1}))^{1\slash 2}$. On the other hand, another integration in polar coordinates
shows that for $k > 0$ the Fourier coefficients $\langle F, \Phi_{n,k,m}^{(d)}\rangle$ vanish
since then $Y_m^k \perp Y_1^0 \equiv \textrm{const.}$ Thus the expansion of $F$ in $B^d$
is in fact the expansion of its profile $f$ with respect to the system
$\{\phi_n^{d\slash 2-1}\}$ on $(0,1)$.
Similarly, the associated heat and subordinated semigroups are also related via the radial case,
as stated below.

\begin{propo} \label{sem_con}
Let $d \ge 1$, $\nu=d\slash 2-1$ and $f \in \spann\{\phi_n^{\nu}: n \ge 1\}$. Then
$$
\big( T_t^{\nu,\alpha}f\big) \circ \varphi(\mathtt{x}) =
	\mathcal{T}_t^{(d),\alpha}(f\circ \varphi)(\mathtt{x}), \qquad \mathtt{x} \in B^d,
$$
where $\varphi(\mathtt{x})=|\mathtt{x}|$.
\end{propo}

\begin{proof}
For $f=\phi_n^{\nu}$ we have
\begin{align*}
\big( T_t^{\nu,\alpha}\phi_n^{\nu}\big) \circ \varphi(\mathtt{x})
& = e^{-t\lambda_{n,\nu}^{\alpha}}\phi_n^{\nu}(|\mathtt{x}|) \\
& = c_d e^{-t\lambda_{n,d\slash 2-1}^{\alpha}} \Phi_{n,0,1}^{(d)}(\mathtt{x})
= c_d \mathcal{T}_t^{(d),\alpha} \Phi_{n,0,1}^{(d)}(\mathtt{x})
= \mathcal{T}_t^{(d),\alpha}\big(\phi_n^{\nu}\circ \varphi\big)(\mathtt{x}),
\end{align*}
where $c_d=(\sigma(S^{d-1}))^{1\slash 2}$. The conclusion follows.
\end{proof}

In fact the semigroups are related in the same way for more general functions $f$.
This is confirmed by the relation between the corresponding integral kernels established below.
\begin{theor} \label{ker_con}
Let $d \ge 1$ and $\nu=d\slash 2-1$. Then
$$
G_t^{\nu,\alpha}(x,y) = \int_{S^{d-1}} \mathcal{G}_t^{(d),\alpha}(\mathtt{x},y\xi)\, d\sigma(\xi),
\qquad x,y \in (0,1), \quad t>0,
$$
where $\mathtt{x}=(x,0,\ldots,0) \in B^d$.
\end{theor}

\begin{proof}
Let $f \in \spann\{\phi_n^{\nu}\}$. Using Proposition \ref{sem_con} and then integrating in polar
coordinates we get
\begin{align*}
\int_0^1 G_t^{\nu,\alpha}(|\mathtt{x}|,y)f(y)y^{2\nu+1}\,dy &
= \int_{B^d} \mathcal{G}_t^{(d),\alpha}(\mathtt{x},\mathtt{y})f(|\mathtt{y}|)\, d\mathtt{y}\\
& = \int_0^1 \int_{S^{d-1}} \mathcal{G}_t^{(d),\alpha}(\mathtt{x},y\xi)\, d\sigma(\xi) \, f(y)y^{d-1}\,dy.
\end{align*}
Since the subspace spanned by the $\phi_n^{\nu}$ is dense in $L^2((0,1),d\mu_{\nu})$
we conclude the desired identity up to an exceptional set of $y$ of measure $0$.
This set, however, must be empty because the kernels involved are continuous with respect
to their arguments.
\end{proof}

Actually, each of the kernels $G_t^{\nu,\alpha}(x,y)$
and $\mathcal{G}_t^{(d),\alpha}(\mathtt{x},\mathtt{y})$
is a jointly smooth function of $t>0$ and its arguments. This can be verified even directly, by
term by term differentiation of the defining series, with the aid of basic bounds for
$d_{n,\nu}$, $\lambda_{n,\nu}$, $J_{\nu}$ (see \cite[Section 2]{OS}) and the fact that
$|Y_m^k(\xi)|$ can be dominated, uniformly in $m$ and $\xi$, by a polynomial in $k$
(see for instance \cite[Chapter IV, Corollary 2.9 (b)]{SW}).
We leave details to interested readers.

\section{Kernel estimates} \label{sec:heat}

Heat kernels in various contexts were extensively investigated in the literature.
In particular, it is well known that the heat kernel corresponding to the Dirichlet Laplacian
in $B^d$, $d \ge 1$, satisfies the bounds (cf. \cite[(1.9.1)]{Da})
\begin{equation} \label{rb}
0 < \mathcal{G}_t^{(d),2}(\mathtt{x},\mathtt{y}) \le \frac{1}{(4\pi t)^{d\slash 2}}
	\exp\bigg({-\frac{|\mathtt{x}-\mathtt{y}|^2}{4t}}\bigg),
	\qquad \mathtt{x},\mathtt{y} \in B^d, \quad t>0.
\end{equation}
These bounds are perhaps most clear from the probabilistic point of view since
$\mathcal{G}_t^{(d),2}(\mathtt{x},\mathtt{y})$ is just the transition probability density of the time-scaled
Brownian motion $B_{2t}$ killed upon leaving $B^d$. However, finding sharp bounds for
$\mathcal{G}_t^{(d),2}(\mathtt{x},\mathtt{y})$ that describe precisely the interplay between
$\mathtt{x},\mathtt{y}$ and $t$, and the boundary behavior, is a much more complicated matter.
The complete and qualitatively sharp estimates are available only recently,
for bounded $C^{1,1}$ domains.
The case of $B^d$, obviously being geometrically simpler, does not seem to have been known earlier.

\begin{theor}[\cite{DS,Da1,Z,Song}] \label{thm:zhang}
Let $d \ge 1$. Given $T>0$, there exists a constant $c>1$ such that
\begin{align}
& \bigg[ \frac{(1-|\mathtt{x}|)(1-|\mathtt{y}|)}{t} \wedge 1 \bigg] \frac{1}{t^{d\slash 2}}
	\exp\bigg({-\frac{c|\mathtt{x}-\mathtt{y}|^2}{t}}\bigg) \nonumber \\ & \qquad \lesssim
	\mathcal{G}_t^{(d),2}(\mathtt{x},\mathtt{y}) \lesssim
	\bigg[ \frac{(1-|\mathtt{x}|)(1-|\mathtt{y}|)}{t} \wedge 1 \bigg] \frac{1}{t^{d\slash 2}}
	\exp\bigg({-\frac{|\mathtt{x}-\mathtt{y}|^2}{ct}}\bigg), \label{hb}
\end{align}
uniformly in $\mathtt{x},\mathtt{y} \in B^d$ and $0<t\le T$. Moreover, if $T$ is chosen sufficiently
large, then
\begin{equation} \label{hbl}
\mathcal{G}_t^{(d),2}(\mathtt{x},\mathtt{y}) \simeq (1-|\mathtt{x}|)(1-|\mathtt{y}|)
	\exp\big({-t \lambda_{1,d\slash 2-1}^2}\big), \qquad \mathtt{x},\mathtt{y} \in B^d, \quad t\ge T.
\end{equation}
\end{theor}
The upper bound in \eqref{hb} as well as the large time behavior \eqref{hbl} has been known at least
since the 1980s, cf. Davies and Simon \cite{DS} and Davies \cite{Da1}, see also the related comments
in \cite{Z}. The existence of the lower bound in \eqref{hb} was shown by Zhang \cite[Theorem 1.1]{Z}
under the assumption $d\ge 3$ and with an implicitly fixed $T$. Later these restrictions were
removed by Song \cite[Theorems 3.8 and 3.9]{Song}. Note that the upper estimate in \eqref{hb}
holds in fact for all $t>0$, which is not the case of the lower bound.

Sharp estimates for the subordinated kernels were found recently by Song \cite{Song}.
\begin{theor}[{\cite[Theorem 4.7]{Song}}] \label{thm:song}
Let $d \ge 1$ and $0<\alpha<2$. Given $T>0$, we have
\begin{equation} \label{sb}
\mathcal{G}_t^{(d),\alpha}(\mathtt{x},\mathtt{y}) \simeq
	\bigg[ \frac{(1-|\mathtt{x}|)(1-|\mathtt{y}|)}{t^{2\slash \alpha}+|\mathtt{x}-\mathtt{y}|^2}
		\wedge 1 \bigg] \frac{t}{(t^{2\slash \alpha}+|\mathtt{x}-\mathtt{y}|^2)^{(d+\alpha)\slash 2}},
\end{equation}
uniformly in $\mathtt{x},\mathtt{y} \in B^d$ and $0<t\le T$.
\end{theor}
Similarly as in \eqref{hb}, the upper bound in \eqref{sb} holds in fact for all $t>0$, which is
not the case of the lower bound.
The expected large time behavior for $0<\alpha<2$ is of course
\begin{equation*}
\mathcal{G}_t^{(d),\alpha}(\mathtt{x},\mathtt{y}) \simeq (1-|\mathtt{x}|)(1-|\mathtt{y}|)
	\exp\big({-t \lambda_{1,d\slash 2-1}^{\alpha}}\big), \qquad \mathtt{x},\mathtt{y} \in B^d, \quad t\ge T,
\end{equation*}
with $T$ chosen sufficiently large.
These bounds can be proved by a direct analysis
of the series defining $\mathcal{G}_t^{(d),\alpha}(\mathtt{x},\mathtt{y})$, similar to that in the proof
of Theorem \ref{thm:large} below. This, however, is beyond the scope of this paper.

We now transfer the above bounds for $\mathcal{G}_t^{(d),\alpha}(\mathtt{x},\mathtt{y})$ to the
Fourier-Bessel setting on the interval $(0,1)$. More precisely, we shall prove the following.
\begin{theor} \label{thm:main}
Let $\nu = d\slash 2-1$ for some $d \ge 1$. Given $T>0$, there exists a constant $c>1$ such that
\begin{align*}
& (xy)^{-\nu-1\slash 2} \bigg(\frac{xy}{t}\wedge 1\bigg)^{\nu+1\slash 2}
	\bigg[ \frac{(1-x)(1-y)}{t} \wedge 1\bigg] \frac{1}{\sqrt{t}}\exp\bigg({-c\frac{(x-y)^2}{t}}\bigg) \\
& \qquad \lesssim G_t^{\nu,2}(x,y) \lesssim	
	(xy)^{-\nu-1\slash 2} \bigg(\frac{xy}{t}\wedge 1\bigg)^{\nu+1\slash 2}
	\bigg[ \frac{(1-x)(1-y)}{t} \wedge 1\bigg] \frac{1}{\sqrt{t}} \exp\bigg({-\frac{(x-y)^2}{ct}}\bigg),
\end{align*}
uniformly in $x,y \in (0,1)$ and $0 < t \le T$. Moreover, for sufficiently large $T$ we have
$$
G_t^{\nu,2}(x,y) \simeq (1-x)(1-y) \exp\big({-t\lambda_{1,\nu}^2}\big),
	\qquad x,y \in (0,1), \quad t \ge T.
$$
\end{theor}

\begin{theor} \label{thm:poisson}
Let $\alpha \in (0,2)$ and $\nu=d\slash 2-1$ for some $d \ge 1$. Given $T>0$, we have
\begin{align*}
& G_t^{\nu,\alpha}(x,y)\\ & \qquad\simeq (xy)^{-\nu-1\slash 2}
	\bigg( \frac{xy}{t^{2\slash \alpha}+(x-y)^2} \wedge 1\bigg)^{\nu+1\slash 2}
	\bigg[ \frac{(1-x)(1-y)}{t^{2\slash \alpha}+(x-y)^2} \wedge 1\bigg]
	\frac{t}{(t^{2\slash \alpha}+(x-y)^2)^{(\alpha+1)\slash 2}},
\end{align*}
uniformly in $x,y \in (0,1)$ and $0 < t \le T$.
\end{theor}

At the end of this section we will complement these results by deriving, directly and
for all $\nu >-1$, sharp large time bounds for the kernels $G_t^{\nu,\alpha}(x,y)$, $0 < \alpha \le 2$.

To prove Theorem \ref{thm:main} we need the following.
\begin{lemma} \label{lem:bes}
Let $d \ge 1$ and $x,y,t,c>0$ be fixed. Then
$$
\int_{S^{d-1}} \exp\bigg({-\frac{|\mathtt{x}-y\xi|^2}{ct}}\bigg)\, d\sigma(\xi) = (2\pi)^{\nu+1}
	 \exp\bigg({-\frac{x^2+y^2}{ct}}\bigg) \Big(\frac{ct}{2xy}\Big)^{\nu} I_{\nu}\Big(\frac{2xy}{ct}\Big),
$$
where $\mathtt{x} = (x,0,\ldots,0)\in \mathbb{R}^d$, $\nu=d\slash 2-1$ and $I_{\nu}$ denotes the
modified Bessel function of the first kind and order $\nu$.
\end{lemma}

\begin{proof}
We first deal with the case $d \ge 2$. Observe that for $\xi \in S^{d-1}$
$$
\exp\bigg({-\frac{|\mathtt{x}-y\xi|^2}{ct}}\bigg) = \exp\bigg({-\frac{x^2+y^2}{ct}}\bigg)
	\exp\Big({\frac{2xy\xi_1}{ct}}\Big).
$$
Thus the integrand is a zonal function depending only on $\xi_1$. In this case the integration over
$S^{d-1}$ reduces to a one-dimensional integral against $d\xi_1$, see for instance the proof
of \cite[Chapter IV, Corollary 2.16]{SW}. Precisely, we have
$$
\int_{S^{d-1}} \exp\Big({\frac{2xy\xi_1}{ct}}\Big) \, d\sigma(\xi)
	= \sigma(S^{d-2}) \int_{-1}^1 \exp\Big({\frac{2xy}{ct} \xi_1}\Big) (1-\xi_1^2)^{(d-3)\slash 2}\, d\xi_1,
$$
where $\sigma(S^{d-2})=2\pi^{(d-1)\slash 2}\slash \Gamma((d-1)\slash 2)$ is the surface area
measure of $S^{d-2}$. To evaluate the last integral we invoke Schl\"afli's Poisson type representation
for the Bessel function $I_{\nu}$ (cf. \cite[Chapter III, Section 3$\cdot$71, (9)]{Wat})
$$
I_{\nu}(z) = \frac{1}{\sqrt{\pi}2^{\nu}\Gamma(\nu+1\slash 2)} z^{\nu} \int_{-1}^1 \exp(zs)
	(1-s^2)^{\nu-1\slash 2}\, ds, \qquad \nu > -1\slash 2.
$$
Putting now all the facts together we arrive at the desired identity.

When $d=1$ the identity is verified directly, taking into account that
(cf. \cite[Chapter III, Section 3$\cdot$71, (10)]{Wat})
$I_{-1\slash 2}(z)=(2\slash (\pi z))^{1\slash 2} \cosh z$.
\end{proof}

\begin{proof}[Proof of Theorem \ref{thm:main}]
The large time behavior is a consequence of Theorem \ref{ker_con} and Theorem
\ref{thm:zhang}. To show the remaining estimates we use in addition Lemma \ref{lem:bes}.
This produces the bounds
\begin{align*}
& (xy)^{-\nu}\bigg[ \frac{(1-x)(1-y)}{t} \wedge 1\bigg] \frac{1}{t} \exp\bigg({-c\frac{x^2+y^2}{t}}\bigg)
	I_{\nu}\Big(\frac{2xy}{ct}\Big)\\ & \qquad \lesssim G_t^{\nu,2}(x,y) \lesssim
	(xy)^{-\nu}\bigg[ \frac{(1-x)(1-y)}{t} \wedge 1\bigg] \frac{1}{t} \exp\bigg({-\frac{x^2+y^2}{ct}}\bigg)
	I_{\nu}\Big(\frac{2xy}{ct}\Big).
\end{align*}
Now the conclusion follows by applying the standard asymptotics for $I_{\nu}$, $\nu > -1$,
$$
I_{\nu}(z) \simeq  z^{\nu}, \qquad z \to 0^+, \qquad \qquad I_{\nu}(z) \simeq z^{-1\slash 2} \exp({z}),
	\qquad z \to \infty,
$$
together with the fact that $I_{\nu}$ is continuous.
\end{proof}

In order to prove Theorem \ref{thm:poisson} we first analyze the behavior of a one-dimensional integral depending
on several parameters.
\begin{lemma} \label{lem:int}
Let $\gamma$ and $\eta$ be such that $\gamma>\eta+1>0$. Then
$$
\int_{-1}^1 \frac{(1-s^2)^{\eta}\, ds}{(D-Bs)(A-Bs)^{\gamma}} \simeq
	\frac{1}{(D-B)A^{\eta+1}(A-B)^{\gamma-\eta-1}}, \qquad 0<B<A<D.
$$
\end{lemma}

\begin{proof}
Observe that
$$
\int_{-1}^1 \frac{(1-s^2)^{\eta}\, ds}{(D-Bs)(A-Bs)^{\gamma}} \simeq
\int_{0}^1 \frac{(1-s^2)^{\eta}\, ds}{(D-Bs)(A-Bs)^{\gamma}} \simeq
\int_{0}^1 \frac{(1-s)^{\eta}\, ds}{(D-Bs)(A-Bs)^{\gamma}},
$$
so it is enough to analyze the last integral, which we further denote by $\mathcal{I}$.
It is convenient to consider the following three cases. Altogether, they cover all admissible
configurations of $A,B,D$.
\newline {\bf Case 1:} $\mathbf{D>A\ge 2B.}$ Since in this case $A-Bs \simeq A \simeq A-B$ and
$D-Bs \simeq D \simeq D-B$, for $s \in (0,1)$, the conclusion is trivial.
\newline {\bf Case 2:} $\mathbf{D\ge 2B> A > B.}$
Now $D-Bs \simeq D \simeq D-B$ for $s \in (0,1)$ and therefore
$$
\mathcal{I} \simeq \frac{1}{D-B}\int_0^1 \frac{(1-s)^{\eta}\, ds}{(A-Bs)^{\gamma}}.
$$
Changing the variable $u = 1-s$ we get
$$
\mathcal{I} \simeq \frac{1}{D-B}\int_0^1 \frac{u^{\eta}\, du}{(A-B + Bu)^{\gamma}}.
$$
Splitting the last integral at the point $u=(A-B)\slash B$ and estimating the resulting integrals
separately we obtain
$$
\int_0^{(A-B)\slash B} \frac{u^{\eta}\, du}{(A-B+Bu)^{\gamma}} \simeq
	\frac{1}{(A-B)^{\gamma}} \int_0^{(A-B)\slash B} u^{\eta}\, du \simeq
	\frac{1}{A^{\eta+1}(A-B)^{\gamma-\eta-1}}
$$
and
$$
\int_{(A-B)\slash B}^1 \frac{u^{\eta}\, du}{(A-B+Bu)^{\gamma}} \lesssim \frac{1}{B^{\gamma}}
\int_{(A-B)\slash B}^1 u^{\eta-\gamma}\, du \lesssim \frac{1}{B^{\gamma}}
	\Big(\frac{A-B}{B}\Big)^{\eta-\gamma+1} \simeq \frac{1}{A^{\eta+1}(A-B)^{\gamma-\eta-1}}.
$$
This implies the desired conclusion.
\newline {\bf Case 3:} $\mathbf{2B>D> A > B.}$
Changing the variable $u=1-s$ we get
$$
\mathcal{I} = \int_0^1 \frac{u^{\eta}\, du}{(D-B+Bu)(A-B+Bu)^{\gamma}}.
$$
We now split the above integral at the points $u_1=(A-B)\slash B$ and $u_2=(D-B)\slash B$ and
estimate the resulting integrals separately. We have
\begin{align*}
\int_0^{(A-B)\slash B} \frac{u^{\eta}\, du}{(D-B+Bu)(A-B+Bu)^{\gamma}} & \simeq
	\frac{1}{(D-B)(A-B)^{\gamma}} \int_0^{(A-B)\slash B} u^{\eta}\, du \\ & \simeq
	\frac{1}{(D-B)A^{\eta+1}(A-B)^{\gamma-\eta-1}}.
\end{align*}
Further,
\begin{align*}
\int_{(A-B)\slash B}^{(D-B)\slash B} \frac{u^{\eta}\, du}{(D-B+Bu)(A-B+Bu)^{\gamma}} & \lesssim
	\frac{1}{(D-B)B^{\gamma}} \int_{(A-B)\slash B}^{(D-B)\slash B} u^{\eta-\gamma}\, du \\
	& \lesssim \frac{1}{(D-B)B^{\gamma}}\Big(\frac{A-B}{B}\Big)^{\eta-\gamma+1} \\
	& \simeq \frac{1}{(D-B)A^{\eta+1}(A-B)^{\gamma-\eta-1}}.
\end{align*}
Finally,
\begin{align*}
\int_{(D-B)\slash B}^1 \frac{u^{\eta}\, du}{(D-B+Bu)(A-B+Bu)^{\gamma}} & \lesssim
	\frac{1}{B^{\gamma+1}} \int_{(D-B)\slash B}^1 u^{\eta-\gamma-1}\, du \\
& \lesssim \frac{1}{B^{\gamma+1}} \Big(\frac{D-B}{B}\Big)^{\eta-\gamma} \\
& \lesssim \frac{1}{(D-B)A^{\eta+1}(A-B)^{\gamma-\eta-1}}.
\end{align*}
The conclusion again follows.
\end{proof}

\begin{proof}[Proof of Theorem \ref{thm:poisson}]
We assume $d\ge 2$; the case $d=1$ is a simplified version of what follows.
According to Theorem \ref{ker_con} we must integrate the right-hand side in \eqref{sb}
with $\mathtt{x}=(x,0,\ldots,0)\in B^d$ and $\mathtt{y}$ replaced by $y\xi$, over $S^{d-1}$ and
with respect to $d\sigma(\xi)$. To do that, we first use the relation
$z \wedge 1 \simeq z\slash (z+1)$, $z >0$, in order to switch to a comparable expression
that is more suitable for the integration. This leads to
$$
G_t^{\nu,\alpha}(x,y) \simeq \int_{S^{d-1}}
	\frac{(1-x)(1-y)}{t^{2\slash \alpha}+|\mathtt{x}-y\xi|^2+(1-x)(1-y)} \,
	\frac{t}{(t^{2\slash \alpha}+|\mathtt{x}-y\xi|^2)^{(d+\alpha)\slash 2}} \, d\sigma(\xi).
$$
Here $|\mathtt{x}-y\xi|^2=x^2+y^2-2xy\xi_1$ so, in fact, the integrated function is zonal.
Thus (see the proof of Lemma \ref{lem:bes})
\begin{align*}
& G_t^{\nu,\alpha}(x,y) \\ & \simeq (1-x)(1-y)t \int_{-1}^1
	\frac{(1-\xi_1^2)^{(d-3)\slash 2}\, d\xi_1}{(t^{2\slash \alpha}+x^2+y^2+(1-x)(1-y)-2xy\xi_1)
		(t^{2\slash \alpha}+x^2+y^2-2xy\xi_1)^{(d+\alpha)\slash 2}}.
\end{align*}
The last integral is handled by Lemma \ref{lem:int} applied with
$A=t^{2\slash \alpha}+x^2+y^2$, $B=2xy$, $D=A+(1-x)(1-y)$, $\gamma = (d+\alpha)\slash 2$
and $\eta = (d-3)\slash 2 = \nu-1\slash 2$. Consequently, we obtain
$$
G_t^{\nu,\alpha}(x,y) \simeq \frac{(1-x)(1-y)t}{(t^{2\slash \alpha}+(x-y)^2+(1-x)(1-y))
	(t^{2\slash \alpha}+x^2+y^2)^{\nu+1\slash 2}(t^{2\slash \alpha}+(x-y)^2)^{(\alpha+1)\slash 2}}.
$$
This, combined with the relations $x^2+y^2 \simeq (x-y)^2+xy$ and
$z\slash (1+z) \simeq z \wedge 1$, $z>0$, finishes the proof.
\end{proof}

We close this section by describing the long time behavior of $G_t^{\nu,\alpha}(x,y)$
for all $\nu>-1$ and $0<\alpha \le 2$.
\begin{theor} \label{thm:large}
Let $\nu>-1$ and $0<\alpha \le 2$. There exists $T>0$ such that
$$
G_t^{\nu,\alpha}(x,y) \simeq (1-x)(1-y) \exp\big(-t \lambda_{1,\nu}^{\alpha}\big), \qquad x,y \in (0,1),
	\quad t \ge T.
$$
Moreover, the upper bound holds for an arbitrary fixed $T>0$.
\end{theor}

\begin{proof}
We decompose, see \eqref{ker},
$$
G_t^{\nu,\alpha}(x,y) = e^{-t\lambda_{1,\nu}^{\alpha}} \phi_1^{\nu}(x)\phi_1^{\nu}(y)
	+ \sum_{n=2}^{\infty} e^{-t\lambda_{n,\nu}^{\alpha}} \phi_n^{\nu}(x)\phi_n^{\nu}(y).
$$
To treat the first term here we note that
$$
\phi_1^{\nu}(x) \simeq (1-x), \qquad x \in (0,1).
$$
Indeed, $\phi_1^{\nu}(x)$ is strictly positive for $x\in (0,1)$, continuous for $x \in [0,1]$
(see \cite[(2.3)]{OScon}), $\phi_1^{\nu}(0)>0$, $\phi_1^{\nu}(1)=0$ and
$\frac{d}{dx} \phi_1^{\nu}(1) = -\sqrt{2}\lambda_{1,\nu}<0$.
Here the last identity follows from the formula
\begin{equation} \label{diff_phi}
\frac{d}{dx} \phi_n^{\nu}(x) = - x^{-\nu-1\slash 2} \lambda_{n,\nu}\, d_{n,\nu}
	(\lambda_{n,\nu}x)^{1\slash 2} J_{\nu+1}(\lambda_{n,\nu}x),
\end{equation}
which is a straightforward consequence of the differentiation rule for $J_{\nu}$, see \cite[(2.1)]{OScon}.

The proof will be finished once we show the bound
$$
\bigg|\sum_{n=2}^{\infty} e^{-t\lambda_{n,\nu}^{\alpha}} \phi_n^{\nu}(x)\phi_n^{\nu}(y)\bigg| \lesssim
	(1-x)(1-y) e^{-t {\lambda_{1,\nu}^{\alpha}}} e^{-t\varepsilon}, \qquad x,y \in (0,1), \quad t \ge T,
$$
for an arbitrary fixed $T>0$ and certain $\varepsilon >0$. To proceed, we first justify the bound
\begin{equation} \label{bound}
\big|\phi_n^{\nu}(x)\big| \lesssim (1-x) n^{\nu+2}, \qquad x \in (0,1), \quad n \ge 1.
\end{equation}
By the Mean Value Theorem, for each $x\in (0,1)$ there exists $\theta \in (x,1)$ such that
$$
\big|\phi_n^{\nu}(x)\big| = \big| \phi_n^{\nu}(x)-\phi_n^{\nu}(1)\big|
	= (1-x) \bigg| \frac{d}{dx} \phi_n^{\nu}(x)\Big|_{x=\theta}\bigg|.
$$
To estimate the last derivative we combine \eqref{diff_phi} with the bounds given in
\cite[Section 2]{OScon}, see \cite[(1.3),(2.5),(2.6)]{OScon}, getting
$$
\bigg| \frac{d}{dx} \phi_n^{\nu}(x)\Big|_{x=\theta}\bigg| \lesssim n \big( 1 \vee n^{\nu+1\slash 2}\big) 
\le n^{\nu+2}, \qquad n \ge 1, \quad \theta \in (0,1).
$$
This gives \eqref{bound}. Now we may write
$$
\sum_{n=2}^{\infty} e^{-t\lambda_{n,\nu}^{\alpha}} \big|\phi_n^{\nu}(x)\phi_n^{\nu}(y)\big| \lesssim
	(1-x)(1-y) e^{-t\lambda_{1,\nu}^{\alpha}} e^{-t \varepsilon}
		\sum_{n=2}^{\infty} n^{2\nu+4} e^{-t(\lambda_{n,\nu}^{\alpha}-\lambda_{1,\nu}^{\alpha}-\varepsilon)}.
$$
Choosing $\varepsilon$ such that $0 < \varepsilon < \lambda_{2,\nu}^{\alpha}-\lambda_{1,\nu}^{\alpha}$
and taking into account that $\lambda_{n,\nu}\simeq n$, $n\to \infty$ (cf. \cite[(2.6)]{OScon}),
we see that the last series is bounded uniformly in $t\ge T$, for any fixed $T>0$. The conclusion follows.
\end{proof}

\section{Comments and remarks} \label{sec:rem}

\subsection*{Elementary cases}
There are only two cases, $\nu=\pm 1\slash 2$, when the Fourier-Bessel system has an explicit form,
i.e. $d_{n,\nu}$, $J_{\nu}$ and $\lambda_{n,\nu}$ are all given explicitly.
We have (see \cite[Section 1]{OS}) $\lambda_{n,-1\slash 2}=\pi (n-1\slash 2)$,
$\lambda_{n,1\slash 2}=\pi n$, and
$$
\phi_n^{-1\slash 2}(x) = \sqrt{2} \cos\big(\pi (n-1\slash 2)x\big), \qquad
\phi_n^{1\slash 2}(x) = x^{-1}\sqrt{2} \sin(\pi n x).
$$
Consequently, the Poisson kernel $G_t^{\nu,1}(x,y)$ can be computed when $\nu=\pm 1\slash 2$.
Indeed, using basic trigonometric identities and the formulas
(cf. \cite[5.4.12 (1),(2)]{PBM})
$$
\sum_{j=1}^{\infty} e^{-Aj}\sin(Bj) = \frac{1}2 \, \frac{\sin B}{\cosh A - \cos B}, \qquad
\sum_{j=1}^{\infty} e^{-Aj}\cos(Bj) = \frac{1}2 \, \frac{\sinh A}{\cosh A - \cos B}-\frac{1}{2},
$$
we arrive at
\begin{align*}
G_t^{-1\slash 2,1}(x,y) & = \frac{\sinh\frac{\pi t}{2}\cos\frac{\pi}2(x-y)}{\cosh \pi t - \cos\pi(x-y)}
 + \frac{\sinh\frac{\pi t}{2}\cos\frac{\pi}2(x+y)}{\cosh \pi t - \cos\pi(x+y)},\\
G_t^{1\slash 2,1}(x,y) & = \frac{1}{2xy}\bigg[
\frac{\sinh \pi t}{\cosh \pi t - \cos \pi(x-y)}-\frac{\sinh \pi t}{\cosh \pi t - \cos \pi(x+y)}\bigg].
\end{align*}
Of course, sharp estimates for $G_t^{\pm 1\slash 2,1}(x,y)$ can be obtained directly from these formulas,
though it is not immediate.

The heat kernel $G_t^{\nu,2}(x,y)$ is `computable' as well for $\nu=\pm 1\slash 2$,
in the sense that the resulting expressions (in fact series) do not contain oscillations.
For the simpler case $\nu=1\slash 2$ we have
$$
G_t^{1\slash 2,2}(x,y) = \frac{1}{xy}\sum_{j \in \mathbb{Z}} \Bigg[ \frac{1}{\sqrt{4\pi t}}
	\exp\bigg({-\frac{(x-y-2j)^2}{4t}}\bigg) - \frac{1}{\sqrt{4\pi t}}
	\exp\bigg({-\frac{(x+y-2j)^2}{4t}}\bigg) \Bigg].
$$
This is essentially a well known Jacobi type identity. The series above represents the heat kernel in the
setting of the system $\{\psi_n^{\nu}\}$ with $\nu=1\slash 2$. Note that the corresponding differential
operator is $\widetilde{L}_{\nu} f(x) = x^{\nu+1\slash 2} L_{\nu} ((\cdot)^{-\nu-1\slash 2}f)(x)$, and
we have $\widetilde{L}_{1\slash 2} = -\Delta$. 
The formula in question can be derived by solving the initial-value
problem for the classical heat equation in the interval $(0,1)$ with Dirichlet boundary conditions, see
\cite[Chapter 3, Exercises 3.5-3.7]{Cannon}. Here the trick relies on extending a function $f$
prescribing initial values on $(0,1)$ to an odd function on $(-1,1)$ and then considering its periodic
extension $\tilde{f}$ to $\mathbb{R}$ with period $2$. The solution of the heat equation on $\mathbb{R}$
with initial values prescribed by $\tilde{f}$, which is given by convolving $\tilde{f}$ with
the Gauss-Weierstrass kernel, provides also the solution of our initial-value problem in the interval
$(0,1)$. The convolution can be written in terms of an integral involving the above series,
due to the symmetries of $\tilde{f}$.

The case $\nu=-1\slash 2$ is slightly more involved, but follows in the same spirit.
Now $L_{\nu}=-\Delta$ and one solves the initial-value problem for the heat equation in $(0,1)$ 
with the Neumann
condition at $x=0$ and the Dirichlet condition at $x=1$. The relevant extension $\tilde{f}$ is obtained
by extending $f$ to $(-1,1)$ as an even function, then taking extension to $(-1,3)$ that is antisymmetric
with respect to $x=1$ (i.e. satisfies $f(x)=-f(2-x)$ for $x\in (-1,3)$), and finally extending this
function from $(-1,3)$ to $\mathbb{R}$ as periodic with period 4. Solving the heat equation with
initial values $\tilde{f}$ and taking into account the symmetries of $\tilde{f}$ one concludes that
\begin{align*}
G_t^{-1\slash 2,2}(x,y) & = \sum_{j\in \mathbb{Z}} \Bigg[ \frac{1}{\sqrt{4\pi t}}
	 \exp\bigg({-\frac{(x-y-4j)^2}{4t}}\bigg) +
	 \frac{1}{\sqrt{4\pi t}} \exp\bigg({-\frac{(x+y-4j)^2}{4t}}\bigg) \\ &  \qquad -
	 \frac{1}{\sqrt{4\pi t}} \exp\bigg({-\frac{(x-y-4(j+\frac{1}2))^2}{4t}}\bigg) -
	 \frac{1}{\sqrt{4\pi t}} \exp\bigg({-\frac{(x+y-4(j+\frac{1}2))^2}{4t}}\bigg) \Bigg].
\end{align*}
The formulas for $G_t^{\pm 1\slash 2,2}(x,y)$ are suitable to study short time behavior since for
that only the terms corresponding to $j=0$ (the case of $\nu=-1\slash 2$) or to $j=0,1$
(the case of $\nu=1\slash 2$) are essential. On the other hand, they indicate that the related analysis
for general $\nu>-1$ is far from being trivial.
We remark that the kernels $G_t^{\pm 1\slash 2,2}(x,y)$ can be also expressed by Jacobi's elliptic
theta functions $\theta_2,\theta_3$ (see \cite[p.\,792]{PBM} for the definitions),
but this representation does not seem to be very useful for our purposes.

\subsection*{Transference}
Next we comment on the transference from the situation of Euclidean balls to the Fourier-Bessel setting
on the interval $(0,1)$. Since the semigroups in both settings are related via the mapping $\varphi$
from Proposition \ref{sem_con}, the same is true for many fundamental operators expressible through
these semigroups. This concerns, for example, maximal operators of the semigroups, fractional
integrals (potential operators) and Laplace transform type multipliers in both settings.
Moreover, since
Lebesgue measure in $B^d$ transforms to the measure $d\mu_{\nu}$, $\nu=d\slash 2-1$, when projecting
via $\varphi$ to $(0,1)$, mapping properties such as $L^p$-boundedness
and weak type estimates of operators defined in the context of Euclidean balls imply the same
kind of mapping properties for the corresponding operators in the Fourier-Bessel setting on $(0,1)$.
Similar concepts of transference are well known also in other settings of classical orthogonal expansions,
for instance in the direction
\emph{Hermite$\longrightarrow$Laguerre}, see \cite{GIT} and \cite[Section 5]{NS4}.

\subsection*{Rough upper bounds}
Another issue to discuss are consequences of the rough estimate \eqref{rb}.
Transferring it by means of Theorem \ref{ker_con} and Lemma \ref{lem:bes}
to the Fourier-Bessel setting on $(0,1)$ we get
\begin{equation} \label{Han}
0 < G_t^{\nu,2}(x,y) \le W_t^{\nu+1\slash 2}(x,y), \qquad x,y \in (0,1), \quad t >0,
\end{equation}
where $\nu=d\slash 2-1$, $d\ge 1$, and
$$
W_t^{\lambda}(x,y) = (xy)^{-\lambda+1\slash 2} \, \frac{1}{2t} \exp\bigg(-\frac{x^2+y^2}{4t}\bigg)
	I_{\lambda-1\slash 2}\Big(\frac{xy}{2t}\Big)
$$
is the heat kernel in the setting of continuous Fourier-Bessel expansions of type $\lambda>-1\slash 2$
on $(0,\infty)$ (the context of the Hankel transform), see \cite{BHNV} or \cite{BCN}.
The bounds \eqref{Han} are easily explained from the probabilistic point of view.
Indeed, $W_t^{\lambda}(x,y)$ is the transition density for the time scaled Bessel process $X_{2t}$
on $(0,\infty)$, see \cite[Appendix I, Section 21]{BS} (here for $\lambda \in (-1\slash 2,1\slash 2)$
the endpoint $x=0$ is assumed to be reflecting), and $G_t^{\lambda-1\slash 2,2}(x,y)$ is the transition
density for $X_{2t}$ killed upon leaving (through $x=1$) the interval $(0,1)$.
Actually, this probabilistic argument works for all $\nu>-1$ and justifies the following.
\begin{propo} \label{prop:Han}
The bounds \eqref{Han} hold for arbitrary $\nu > -1$.
\end{propo}
This result has several interesting consequences. One of them is, see \cite[Section 6]{NS5}, the fact
that the semigroups $\{T_t^{\nu,\alpha}\}_{t>0}$, $\nu>-1$, $0<\alpha\le 2$, are submarkovian symmetric
diffusion semigroups (to get this for $\alpha<2$ we use the subordination principle).
Since, in view of the probabilistic interpretation, both inequalities in \eqref{Han} are strict,
these semigroups are not Markovian.
Similar conclusions are valid for the corresponding semigroups in the setting of the system
$\{\psi_n^{\nu}\}$, with the restriction $\nu \in \{-1\slash 2\}\cup [1\slash 2,\infty)$,
see \cite[Proposition 6.1]{NS5}.

\subsection*{Maximal operators}
Further important consequences of Proposition \ref{prop:Han} concern mapping properties
of the maximal operators
$$
T_*^{\nu,\alpha}f = \sup_{t>0} \big| T_t^{\nu,\alpha}f\big|.
$$
Since $T_*^{\nu,2}$ is controlled by the analogous maximal operator $W_*^{\nu+1\slash 2}$
in the Bessel setting on $(0,\infty)$ related to the measure $d\mu_{\nu}$, it inherits
mapping properties of $W_*^{\nu+1\slash 2}$. The same is true in the context of the system
$\{\psi_n^{\nu}\}$ and the Bessel setting on $(0,\infty)$ related to Lebesgue measure.
These observations allow to transmit positive parts of \cite[Theorem 2.1]{BHNV} to the
Fourier-Bessel setting.
\begin{theor} \label{thm:max}
Let $\nu > -1$, $1\le p < \infty$, $\delta \in \mathbb{R}$.
Then the maximal operator $T_*^{\nu,2}$, considered on the measure space
$((0,1),x^{\delta}dx)$, has the following mapping properties:
\begin{itemize}
\item[(a)]
$T_*^{\nu,2}$ is of strong type $(p,p)$ if
$p>1 \; \textrm{and} \; -1<\delta<(2\nu+2)p-1;$
\item[(b)]
$T_*^{\nu,2}$ is of weak type $(p,p)$ if
$-1<\delta<(2\nu+2)p-1 \; \textrm{or} \;\;  \delta=2\nu+1;$
\item[(c)]
$T_*^{\nu,2}$ is of restricted weak type $(p,p)$ if
$-1<\delta\le(2\nu+2)p-1.$
\end{itemize}
Moreover, $T_*^{\nu,2}$ is of strong type $(\infty,\infty)$.
\end{theor}
Analogous results, with appropriate adjustments, hold also in the setting of $\{\psi_n^{\nu}\}$,
see \cite[Remark 3.2]{BHNV}. Boundedness properties with more general weights, and even in
tensor product multi-dimensional Fourier-Bessel settings, can be deduced in a similar manner
from the corresponding results in \cite{BCN} and \cite{CSz}; we leave details to interested readers.
The above consequences complement and extend previous results on the heat semigroup maximal operator
related to the system $\{\psi_n^{\nu}\}$ and obtained in \cite{Gilbert} and \cite{CRae}.
Note that, in view of the subordination principle,
the abovementioned mapping properties are inherited by the maximal operators of the subordinated semigroups.
Note also that Theorem \ref{thm:max} is not optimal in the sense that it does not take into account
the boundary behavior at the right endpoint of $(0,1)$, which in comparison to the continuous Bessel case
is improved by the decay of the heat kernel, see Theorem \ref{thm:main}.
Actually, this lack of optimality pertains to many other weighted results obtained so far for various
operators in the Fourier-Bessel setting.

The results of this paper shed some new light to the problems posed in \cite{Bet} in which the
heat semigroup comes into play. In particular, in connection with the question \cite[Q.6]{Bet},
it turns out that there is a pencil phenomenon associated with the heat semigroup maximal
operator related to the system $\{\psi_n^{\nu}\}$, and its nature is the same as in the case
of continuous Fourier-Bessel expansions in the Lebesgue measure setting. Moreover, Theorem \ref{thm:main}
enables to establish a sharp description of the pencil phenomenon in the Fourier-Bessel framework
when the type index is half-integer.

\subsection*{Final conjecture}
We close the paper with the following natural conjecture. Proving (or disproving) it is
an interesting and important open problem in the theory of Fourier-Bessel expansions.
\begin{conj}
The estimates from Theorem \ref{thm:main} and Theorem \ref{thm:poisson} hold for all $\nu>-1$.
\end{conj}

\end{document}